\documentclass[11pt]{amsart}

\usepackage{amsmath,amssymb,amsthm,accents}
\usepackage{ascmac}
\usepackage{amscd}
\usepackage[initials,short-journals,non-sorted-cites]{amsrefs}
\usepackage{enumitem}
\usepackage{type1cm}
\usepackage{float}
\usepackage[all]{xy}
\usepackage{amsfonts}

\usepackage[dvipdfmx]{graphicx}
\usepackage{comment}
\usepackage{bm}
\usepackage{url}
\usepackage{booktabs}

\usepackage{color}

\allowdisplaybreaks[4]

\newtheorem{theo}{Theorem}[section]
\newtheorem{prop}[theo]{Proposition}
\newtheorem{lemm}[theo]{Lemma}
\newtheorem{cor}[theo]{Corollary}

\theoremstyle{definition}

\newtheorem{remark}[theo]{Remark}

\newcommand{\Z}{\mathbb{Z}}
\newcommand{\N}{\mathbb{N}}
\newcommand{\F}{\mathbb{F}}
\newcommand{\R}{\mathbb{R}}

\title{The second quandle homology group of the knot $n$-quandle}
\author{Kokoro Tanaka}
\address{\scriptsize Department of Mathematics, Tokyo Gakugei University, Nukuikita 4-1-1, Koganei, Tokyo 184-8501, Japan}
\email{kotanaka@u-gakugei.ac.jp}

\author{Yuta Taniguchi}
\address{\scriptsize Department of Mathematics, Graduate School of Science, Osaka University, 
1-1, Machikaneyama, Toyonaka, Osaka, 560-0043, Japan}
\email{yuta.taniguchi.math@gmail.com}

\date{}

\begin{document}

\maketitle

\begin{abstract}
The knot quandle is a complete invariant for oriented classical knots 
in the $3$-sphere 
up to orientation. 
Eisermann computed the second quandle homology group of the knot quandle and showed that it characterizes the unknot. 
In this paper, we 
compute the second quandle homology group of the knot $n$-quandle completely, 
where the knot $n$-quandle is a certain quotient of the knot quandle for each integer $n$ greater than one. 
Although the knot $n$-quandle is weaker than the knot quandle, 
the second quandle homology group of the former is found to have more information than that of the latter.
As one of the consequences, it follows that the second quandle homology group of the knot $3$-quandle characterizes the unknot, the trefoil and the cinquefoil. 
%
\end{abstract}

\section{Introduction}\label{sect:intro}

Quandles \cite{Joyce1982quandle, Matveev1982distributive} are algebraic structures whose axioms correspond to Reidemeister moves for classical knot diagrams. They are compatible with knot theory and various knot invariants have been defined using them. 
%
%
Among such invariants, the knot quandle of an oriented classical knot, 
which is also called an oriented $1$-knot for short, 
in the $3$-sphere is very powerful and known to be a complete invariant 
up to orientation. 
However, it is difficult to treat the knot quandle $Q(K)$ of an oriented $1$-knot $K$ itself, and hence it is necessary to extract information from $Q(K)$. When extracting information, extrinsic information of $Q(K)$ such as the set of all quandle homomorphisms from $Q(K)$ to a chosen fixed finite quandle are often examined, but in this paper we focus on intrinsic information of $Q(K)$ such as quandle homology groups of $Q(K)$ and the knot $n$-quandle obtained from $Q(K)$ for each integer $n$ greater than $1$. 

The homology theory for quandles was developed in \cite{Carter2003quandle} to define knot invariants. 
The first quandle homology group of any connected quandle, such as the knot quandle $Q(K)$ of an oriented $1$-knot $K$, is known to be isomorphic to the infinite cyclic group $\Z$.  
Eisermann \cite{Eisermann2003unknot} computed the second quandle homology group $H^Q_2(Q(K))$ of $Q(K)$ and showed that it characterizes the unknot. More precisely, he showed that  $H^Q_2(Q(K))$ is trivial 
if 
$K$ is trivial, 
and that $H^Q_2(Q(K))$ is group isomorphic to 
$\Z$ 
if 
$K$ is nontrivial.

The knot $n$-quandle $Q_n(K)$ of an oriented $1$-knot $K$ is defined by taking a certain quotient of $Q(K)$ for each integer $n>1$. 
%
Joyce \cite{Joyce1982quandle} introduced the knot quandle $Q(K)$ of $K$ and obtained basic results such as ``complete invariance" for oriented $1$-knots. In the last chapter of \cite{Joyce1982quandle}, the knot $2$-quandle $Q_2(K)$ of $K$ was also introduced. In that chapter, the knot $2$-quandles for some concrete examples were computed, but properties of $Q_2(K)$ remained unknown. Winker \cite{Winker1984quandles} introduced the knot $n$-quandle $Q_n(K)$ of $K$ as a generalization of $Q_2(K)$ for each integer $n>1$, and found a relation between $Q_n(K)$ and the fundamental group of the $n$-fold cyclic branched covering space $M_K^n$ branched over $K$. Hoste and Shanahan \cite{Hoste2017links} further deepened Winker's observation and showed that the finiteness of $Q_n(K)$ is equivalent to that of $\pi_1(M_K^n)$. It follows, for example, that the knot $2$-quandle of any $2$-bridge knot is always finite. On the other hand, the knot quandle of any nontrivial $1$-knot is known to be always infinite as in Subsection~\ref{subsect:type_n-quandle}, 
and hence $Q_n(K)$ can be considered a more tractable invariant than $Q(K)$. 

In this paper, we investigate how much information the knot $n$-quandle $Q_n(K)$ still has 
about $K$. 
More specifically, we compute the second quandle homology group $H_2^Q(Q_n(K))$ of the knot $n$-quandle $Q_n(K)$ completely in Subsection~\ref{subsect:compute}, and reveal information that $H_2^Q(Q_n(K))$ has about $K$ in Subsection~\ref{subsect:detect}. Although $Q_n(K)$ is weaker than $Q(K)$, the second quandle homology group of $Q_n(K)$ is found to have more information about $K$ than that of $Q(K)$ for $n=3,4$ and $5$.  
Here is a remarkable consequence 
of our computations. 

\setcounter{section}{5}\setcounter{theo}{15}
\begin{theo}
The second quandle homology group of the knot $n$-quandle characterizes 
$(1)$ the unknot, the trefoil and the cinquefoil for $n=3$, 
$(2)$ the unknot and the trefoil for $n=4$ or $5$,  
$(3)$ the unknot for $n>5$, and  
$(4)$ the unknot up to $2$-bridge knot summands for $n=2$. 
\end{theo}
\setcounter{section}{1}\setcounter{theo}{0}

\subsection{Strategy to compute}

We outline here our strategy to compute the second quandle homology group of a connected quandle,  
such as the knot $n$-quandle $Q_n(K)$ of an oriented $1$-knot for each integer $n>1$. 
Instead of computing according to the original definition in \cite{Carter2003quandle}, we use 
the quandle extention theory developed 
in \cite{Carter2003extension, Eisermann2003unknot, Eisermann2014quandle}
and the quandle covering theory developed 
in \cite{Eisermann2003unknot, Eisermann2014quandle}. 
For a connected quandle $X$, 
let $\Lambda \curvearrowright \tilde{X} \to X$ be an extension of $X$ by a group $\Lambda$ consisting of a group action $\Lambda \curvearrowright \tilde{X}$ on a quandle $\tilde{X}$ and a surjective quandle homomorphism $\tilde{X} \to X$. 
Now suppose 
that the surjection $\tilde{X} \to X$ is a universal covering. 
Then it follows from \cite{Eisermann2014quandle}
that $H_2^Q(X)$ is group isomorphic to the abelianization $\Lambda^{\mathrm{ab}}$of $\Lambda$. 
More precisely, we have 
that $H_2^Q(X)$ is group isomorphic to $\pi_1(X)^{\mathrm{ab}}$ 
and that $\pi_1(X)$ is group isomorphic to $\Lambda$, 
where $\pi_1(X)$ is the fundamental group \cite{Eisermann2014quandle} of the connected quandle $X$.  
We note that the first statement follows from 
\cite[Theorem~1.15]{Eisermann2014quandle} and 
the second statement follows from 
\cite[Proposition~5.9]{Eisermann2014quandle}.

With these in mind, 
in order to compute $H_2^Q(Q_n(K))$,
we will focus on 
a certain fibered $2$-knot, called the $n$-twist spin of $K$ and denoted by $\tau^n K$, 
which is an embedded $2$-sphere in the $4$-sphere defined by Zeeman \cite{Zeeman1965twisting}. 
The fiber of $\tau^n K$ is known to be the once punctured closed $3$-manifold $M^n_K$, 
where $M^n_K$ is the $n$-fold cyclic branched covering of $S^3$ branched over $K$. 
Let $A^n_K$ be a cyclic group generated by a lift of the longitude of $K$ in $\pi_1(M^n_K)$. 
Then we will show that 
there is an extension $A^n_K \curvearrowright Q(\tau^n K) \to Q_n(K)$ 
in Theorem~\ref{theo:central_extension}, 
and that 
$Q(\tau^n K) \to Q_n(K)$ is a universal covering  
in Theorem~\ref{theo:universal_covering}. 
Combining these two theorems, 
we conclude that $H_2^Q(Q_n(K))$ is group isomorphic to the cyclic group $A^n_K$. 
Finally, by geometric arguments, 
we can determine 
the order of the cyclic group $A^n_K$
explicitly in Section~\ref{sect:quandle_homology}. 

\subsection{Remarks} 
The aforementioned theorems, 
Theorem~\ref{theo:central_extension} and~\ref{theo:universal_covering},  
have their own interests. 
We mention a relation with \cite{Eisermann2003unknot} and that with \cite{Inoue2021knot} below.

Eisermann \cite[Theorem~35]{Eisermann2003unknot} showed that, 
for an oriented nontrivial  $1$-knot $K$,  
there is an extension $\Z \curvearrowright Q(\widehat{K}) \to Q(K)$ 
coming from the long knot $\widehat{K}$ obtained by cutting $K$ open, 
and also that $Q(\widehat{K}) \to Q(K)$ is a universal covering. 
Since the knot quandle of $\widehat{K}$ is naturally isomorphic to that of 
the $0$-twist spin $\tau^0 K$ and the ``knot $0$-quandle'' is the same as 
the knot quandle by definition, 
his extension is rephrased as $\Z \curvearrowright Q(\tau^0 K) \to Q_0(K)$. 
There is no definition for  
a ``space $M^0_K$,'' but a ``group $\pi_1(M^0_K)$''  should be group isomorphic to the commutator subgroup of the knot group of $K$; see the first paragraph of Subsection~\ref{subsect:structure_n-quandle}. Hence it is natural to take a space $M^0_K$ as the infinite cyclic covering space of $S^3\setminus K$, 
and then a cyclic group $A^0_K$ generated by a lift of the longitude of $K$ in $\pi_1(M^0_K)$ is group isomorphic to $\Z$.  
%
Thus, substituting $0$ for $n$ in our two theorems would yield his results formally; 
however, it does not give an alternative proof.

Inoue \cite{Inoue2021knot} introduced the notion of the Schl\"{a}fli quandle $X_n$ 
that is related to the regular tessellation $\{3,n\}$ of geometric $2$-spaces in the sense of the Schl\"{a}fli symbol. 
%
Then, for the trefoil $3_1$, he showed in \cite[Theorem~6.3]{Inoue2021knot}
that there is an extension 
$A_n \curvearrowright Q(\tau^n \, 3_1) \to X_n$ by a cyclic group $A_n$. 
On the other hand, 
when we take an oriented $1$-knot $K$ as the trefoil $3_1$ in Theorem~\ref{theo:central_extension}, 
we have the extension $A^n_{3_1} \curvearrowright Q(\tau^n \, 3_1) \to Q_n(3_1)$. 
Since the quandle $X_n$ turns out to be naturally isomorphic to $Q_n(3_1)$ in Lemma~\ref{lemm:Schlafli}, 
we see that the two extensions are essentially the same, 
and in particular, we have that $A_n$ is group isomorphic to $A^n_{3_1}$. 
He computed $A_n$ for $n \leq 5$ and asked what $A_n$ is for $n>5$. 
As a byproduct, we can answer his question; 
in fact, 
we have that $A_n$ is group isomorphic to $\Z$ for $n>5$.

\subsection{Organization}
This article is organized as follows. In Section~\ref{sect:quandle}, we review basics of quandle theory including the knot quandles and the knot $n$-quandles. In Section~\ref{sect:central_extension}, we study algebraic structures of the knot $n$-quandle and prove that the knot quandle of the $n$-twist spin of $K$ is a central extension of the knot $n$-quandle of $K$ (Theorem~\ref{theo:central_extension}). In Section~\ref{sect:universal_covering}, we review the covering theory for quandles and prove 
that the quandle covering associated with the central extension is a universal covering (Theorem~\ref{theo:universal_covering}). Section~\ref{sect:quandle_homology} is devoted to computing the second quandle homology group of the knot $n$-quandle. Finally, in Section~\ref{sect:topics}, 
we collect three related results: a relation with an extension in \cite{Inoue2021knot}, the types of the knot $n$-quandles, and the cardinalities of the finite knot $n$-quandles.   

\section{Quandles}
\label{sect:quandle}
A \textit{quandle} \cite{Joyce1982quandle,Matveev1982distributive} is a non-empty set $X$ with a binary operation $\ast$ that satisfies the following conditions:
\begin{itemize}
\item For any $x\in X$, we have $x\ast x=x$.
\item For any $y\in X$, the map $S_y:X\to X;x\mapsto x\ast y$ is bijective.
\item For any $x,y,z\in X$, we have $(x\ast y)\ast z=(x\ast z)\ast(y\ast z)$.
\end{itemize} 
We denote $S_y^n(x)$ by $x\ast^n y$ for any $x,y\in X$ and $n\in\Z$.  
Then we set 
$\textrm{type}(X):=\textrm{min}\{ n \in \N \mid x\ast^n y=x\ \textrm{for any }x,y\in X\},$ 
where 
$\textrm{min}\ \emptyset=\infty$. 
We call $\textrm{type}(X)$ the \textit{type of a quandle $X$}.

A map $f:X\to X'$ between quandles is said to be a \textit{quandle homomorphism} if $f(x\ast y)=f(x)\ast f(y)$ for any $x,y\in X$. A quandle homomorphism $f:X\to X'$ is said to be a \textit{quandle isomorphism} if $f$ is a bijection.

The \textit{associated group} of $X$, denoted by $\textrm{As}(X)$, is the group generated by the elements of $X$ subject to the relations $x\ast y=y^{-1}xy$ for $x,y\in X$. 
We note that ${\rm As}$ is a functor from the category of quandles to the category of groups. The associated group ${\rm As}(X)$ acts on $X$ from the right by $x\cdot y:=x\ast y$ for any $x,y\in X$. A quandle $X$ is \textit{connected} if the action of $\textrm{As}(X)$ on $X$ is transitive.

Let $G$ be a group and $\varphi:G\to G$ a group automorphism. We define the binary operation $\ast$ on $G$ by $x\ast y:=\varphi(xy^{-1})y$. Then $\textrm{GAlex}(G,\varphi)=(G,\ast)$ becomes a quandle, which is called the \textit{generalized Alexander quandle}.

Let $N(K)$ be a tubular neighborhood of $K$ and $E(K)=S^{3}\backslash {\rm int}N(K)$ an exterior of an oriented $1$-knot $K$ in the $3$-sphere $S^3$. We fix a base point $p\in E(K)$. Let $Q(K,p)$ be the set of homotopy classes of all paths in $E(K)$ from a point in $\partial E(K)$ to the base point $p$. The set $Q(K,p)$ is a quandle with an operation defined by $[\alpha]\ast[\beta]:=[\alpha\cdot\beta^{-1}\cdot m_{\beta(0)}\cdot \beta]$, where $m_{\beta(0)}$ is a meridian loop starting from $\beta(0)$ and going along in the positive direction. We call $Q(K,p)$ the {\it knot quandle} of $K$, 
which is known to be connected. 
Since the quandle isomorphism class of $Q(K,p) $ does not depend on the base point $p$, we denote the knot quandle $Q(K,p)$ simply by $Q(K)$. Note that, for an oriented $1$-knot $K$, the associated group ${\rm As}(Q(K))$ is group isomorphic to the knot group $G(K):=\pi_1(E(K))$.

Let $K$ be an oriented $1$-knot and $n$ a positive integer. 
We consider a equivalence relation, denoted by $\sim_n$, on the knot quandle $Q(K)$ generated by $x\sim x\ast^n y$ for all $x,y\in Q(K)$. 
Then we define $Q_n(K)$ by the quotient set $Q(K)/\sim_n$. This quotient set $Q_n(K)$ has a quandle operation inherited from the knot quandle $Q(K)$. We call this quandle $Q_n(K)$ the \textit{knot $n$-quandle} of $K$, which is also known to be connected.  In some papers \cite{Crans2019finite,Hoste2017links,Winker1984quandles}, the knot $n$-quandle of $K$ is called the \textit{$n$-quandle} of $K$. Joyce \cite{Joyce1982quandle} used the term ``$n$-quandle" to mean a quandle 
whose type is a divisor of $n$. 
To avoid confusion, we use the term ``knot $n$-quandle of $K$" for $Q_n(K)$ in this paper.

\section{Central extensions of the knot $n$-quandles}
\label{sect:central_extension}
We discuss algebraic structures of the knot $n$-quandle $Q_n(K)$ of an oriented $1$-knot $K$ in Subsection~\ref{subsect:structure_n-quandle}, and show that the knot quandle of the $n$-twist spin of $K$ is a central extension of $Q_n(K)$
in Subsection~\ref{subsect:central_extension}.

\subsection{Algebraic structures of the knot $n$-quandles}
\label{subsect:structure_n-quandle}


For an oriented $1$-knot $K$, 
we denote by $m_K\in G(K)$ a meridian of $K$ and by $l_K\in G(K)$ the preferred longitude of $K$, where $G(K)$ is the knot group of $K$. 
Let $M^n_K$ be the $n$-fold cyclic branched covering of $S^3$ branched over $K$ for an integer $n$ greater than $1$. Then the group $\pi_1(M^n_K)$ is group isomorphic to the quotient group 
$\textrm{Ker}(G(K)\to\Z/n\Z;m_K\mapsto 1)/\langle\langle m_K^n\rangle\rangle$, where $\langle\langle m^n_K\rangle\rangle$ is the normal closure of $m^n_K$. 
Since $l_K$ is an element of $\textrm{Ker}(G(K)\to\Z/n\Z;m_K\mapsto 1)$, 
it can be regarded as an element of $\pi_1(M^n_K)$. 
%
We notice that $l_K\in\pi_1(M^n_K)$ is represented by the branching set of $M^n_K$, 
where this fact will be used in Subsection~\ref{subsec:case_prime_knot}.

We define the group $G_n(K)$ by $G(K)/\langle\langle m^n_K\rangle\rangle$.
It holds that the abelianization of $G_n(K)$ is $\Z/n\Z$, and 
it follows from \cite[Theorem~5.22]{Winker1984quandles} that we have an exact sequence
\[
1\to \pi_1(M^n_K)\to G_n(K)\xrightarrow{\textrm{abelianization}} \Z/n\Z\to 1 . 
\] 
Since this sequence splits, the group $G_n(K)$ is group isomorphic to the semidirect product $\pi_1(M^n_K)\rtimes\Z/n\Z$. Let $\varphi$ be the group automorphism of $\pi_1(M^n_K)$ which is the restriction of the group automorphism of $G_n(K)$ 
sending $g$ to $m_K^{-1}gm_K$ for any $g$ in $G_n(K)$. 
Since $m_K$ commutes with $l_K$ in $G(K)$, we have $\varphi(l_K)=l_K$.

Let $P$ be the subgroup of $G_n(K)$ generated by $m_K$ and $l_K$. Let us define the quandle operation $\ast$ on $P\backslash G_n(K)$ by $Px\ast Py:=Pxy^{-1}m_Ky$ for $Px,Py\in P\backslash G_n(K)$. 
Hoste and Shanahan \cite{Hoste2017links} showed that the knot $n$-quandle $Q_n(K)$ is quandle isomorphic to $P\backslash G_n(K)$. 

Let $A^n_K$ be the subgroup of $\pi_1(M^n_K)$ generated by $l_K$. Then we can define the quandle operation $\ast$ on $A^n_K\backslash \pi_1(M^n_K)$ by $A^n_Kx\ast A^n_Ky:=A^n_K\varphi(xy^{-1})y$ for any $A^n_Kx,A^n_Ky\in H\backslash \pi_1(M^n_K)$. 
\begin{prop}
\label{prop:knot_n_qdle}
The knot $n$-quandle $Q_n(K)$ of an oriented $1$-knot $K$ is quandle isomorphic to $A^n_K\backslash \pi_1(M^n_K)$ for an integer $n>1$.
\end{prop}
\begin{proof}
For notational simplicity, we denote $m_K, l_K$ and $A^n_K$ by $m,l$ and $A$, respectively. 
It is sufficient to show that 
$P\backslash G_n(K)$ and $A \backslash \pi_1(M^n_K)$ are quandle isomorphic. Since $G_n(K)$ is group isomorphic to $\pi_1(M^n_K)\rtimes\Z/n\Z$, for any $x\in G_n(K)$, there exist unique $i_x\in\{0,\ldots,n-1\}$ and $g_x\in\pi_1(M^n_K)$ such that $x=m^{i_x}g_x$. 
Define a map $\Phi \colon P\backslash G_n(K) \to A \backslash \pi_1(M^n_K)$ 
by $\Phi(Px):=A g_x$ for each $Px\in P\backslash G_n(K)$. 
%
If $Px=Py$, we have $x=m^il^jy$ for some $i,j$. Hence, it holds that 
\[
x=m^i l^j y=m^i l^j m^{i_y} g_y=m^i m^{i_y} l^j g_y=m^{i+i_y} l^j g_y.
\]
Thus, we have $\Phi(Px)=A l^j g_y=A g_y=\Phi(Py)$, which implies that 
the map $\Phi$ is well-defined. 

Next, 
define a map $\Psi \colon A \backslash \pi_1(M^n_K)\to P\backslash G_n(K)$ by 
$\Psi(A g):=Pg$ for each $A g\in A \backslash \pi_1(M^n_K)$. 
It is easily seen that 
the map $\Psi$ 
is well-defined. By direct computations, it holds that $\Psi\circ\Phi={\rm Id}$ and $\Phi\circ\Psi={\rm Id}$. 

For any $A g_1,A g_2\in A \backslash \pi_1(M^n_K)$, we have
\begin{eqnarray*}
\Psi(A g_1\ast A g_2)&=&\Psi(A \varphi(g_1g_2^{-1})g_2)\\
&=&P\varphi(g_1g_2^{-1})g_2\\
&=&Pm^{-1}g_1g_2^{-1}mg_2\\
&=&Pg_1g_2^{-1}m g_2\\
&=&Pg_1\ast Pg_2=\Psi(A g_1)\ast\Psi(A g_2),
\end{eqnarray*}
which implies that $\Psi$ is a quandle isomorphism.
\end{proof}

\subsection{Central extensions of the knot $n$-quandles}
\label{subsect:central_extension}
A quandle $\tilde{X}$ is an {\it extension} \cite{Eisermann2003unknot} of a quandle $X$ by a group $\Lambda$ if there is a surjective quandle homomorphism $p:\tilde{X}\to X$ and a left group action $\Lambda \curvearrowright \tilde{X}$ 
such that the following conditions are satisfied: 
\begin{itemize}
\item[(E1)] For any $\tilde{x},\tilde{y}\in\tilde{X}$ and $\lambda\in\Lambda$, we have $\lambda\cdot(\tilde{x}\ast\tilde{y})=(\lambda\cdot\tilde{x})\ast\tilde{y}$ and $\tilde{x}\ast(\lambda\cdot\tilde{y})=\tilde{x}\ast\tilde{y}$.
\item[(E2)] For any $x\in X$, the group $\Lambda$ acts on the preimage $p^{-1}( x)$ freely and transitively.
\end{itemize}
If $\Lambda$ is abelian, we call $\tilde{X}$ a \textit{central extension} of $X$ by the group $\Lambda$.

A \textit{2-knot} $F$ is a 2-sphere embedded in $S^4$, and is \textit{fibered} if its complement $S^4\setminus F$ is a fiber bundle over $S^1$. 
For an oriented 2-knot $F$, we can define the knot quandle of $F$, denoted by $Q(F)$, in a similar manner to the knot quandle of an oriented $1$-knot. Inoue \cite{Inoue2019fibered} showed that the knot quandle of an oriented fibered 2-knot $F$ is quandle isomorphic to $\textrm{GAlex}(\pi_1(M),\psi)$, where $M$ is a fiber of the complement $S^4\setminus F$ and $\psi$ is the group automorphism of $\pi_1(M)$ induced by the monodromy of $S^4\setminus F$. 

For an oriented 1-knot $K$ and an integer $n$, Zeeman \cite{Zeeman1965twisting} introduced the oriented $2$-knot $\tau^nK$, which is called the \textit{$n$-twist spun} $K$ or the \textit{$n$-twist spin} of $K$. Furthermore, for each positive integer $n$, he showed that $\tau^n K$ is a fibered 2-knot whose fiber is the once punctured $M^n_K$ and whose monodromy is the canonical generator of the covering transformation group of $M^n_K$. We note that the group automorphism of $\pi_1(M^n_K)$ induced by the monodromy coincides with $\varphi$ defined in Subsection~\ref{subsect:structure_n-quandle}.
The aim of this section is to show the following for an oriented $1$-knot $K$ and an integer $n>1$. 
\begin{theo}
\label{theo:central_extension}
The knot quandle $Q(\tau^nK)$ of the oriented $2$-knot $\tau^n K$ is a central extension of the knot $n$-quandle $Q_n(K)$ of an oriented $1$-knot $K$ by the cyclic group $A^n_K$ for an integer $n>1$.
\end{theo}
\begin{proof}
Let $p:\pi_1(M^n_K)\to A^n_K\backslash\pi_1(M^n_K)$ be the quotient map. 
By \cite[Theorem~3.1]{Inoue2019fibered}, the quandle $Q(\tau^nK)$ 
is quandle isomorphic to $\textrm{GAlex}(\pi_1(M^n_K),\varphi)$. Hence, by Proposition~\ref{prop:knot_n_qdle}, the map $p$ is a surjective quandle homomorphism from $Q(\tau^nK)$ to $Q_n(K)$. We define the action of $A^n_K$ on $\pi_1(M^n_K)$ by $l_K^i\cdot x=l_K^ix$ for any $i\in \Z$ and $x\in\pi_1(M^n_K)$. 
Obviously, the condition $(\textrm{E2})$ is satisfied. Since it holds that $\varphi(l_K)=l_K$, we have 
\[
(l_K\cdot x)\ast y=\varphi(l_Kxy^{-1})y=l_K\varphi(xy^{-1})y=l_K\cdot(x\ast y)
\]
and
\[
x\ast(l_K\cdot y)=\varphi(x(l_Ky)^{-1})(l_Ky)=\varphi(xy^{-1}l_K^{-1})l_Ky=\varphi(xy^{-1})y=x\ast y
\]
for each $x,y\in \pi_1(M^n_K)=Q(\tau^nK)$. This implies that the condition $(\textrm{E1})$ is also satisfied. 
\end{proof}
\begin{remark}
If $l_K$ is the identity element in $\pi_1(M^n_K)$, then the knot quandle of $\tau^nK$ is quandle isomorphic to the knot $n$-quandle of $K$.
\end{remark}

\section{Universal coverings of the knot $n$-quandles}
\label{sect:universal_covering}
We review covering theory for quandles 
and the fundamental groups of quandles in Subsection~\ref{subsect:def_covering_and_simply_connect}. 
We investigate a relation between the types of quandles and quandle coverings in Subsection~\ref{subsect:type}. 
Then, in Subsection~\ref{subsect:universal_covering_and_simply_connect}, we show the surjective quandle homomorphism $Q(\tau^n K)\to Q_n(K)$ associated with the central extension $A^n_K \curvearrowright Q(\tau^n K) \to Q_n(K)$ 
is a universal covering (Theorem~\ref{theo:universal_covering}).

\subsection{Covering theory for quandles}
\label{subsect:def_covering_and_simply_connect}
A quandle homomorphism $p:\tilde{X}\to X$ between quandles is a \textit{covering} if $p$ is surjective and 
%
$p(\tilde{x})=p(\tilde{y})$ implies $\alpha\ast\tilde{x}=\alpha\ast\tilde{y}$ 
for any $\alpha,\tilde{x},\tilde{y}\in\tilde{X}$.
%
%
This property allows to define an action of $X$ on $\tilde{X}$ by $\alpha * x:=\alpha * \tilde{x}$ with $\tilde{x} \in p^{-1}(x)$ for $x \in X$ and $\alpha \in \tilde{X}$, where this action will be used in the proof of Theorem~\ref{theo:simply_connect_twist_spun}.  
Here is an important example of a covering. For an extension $\tilde{X}$ is of a quandle $X$ by a group $\Lambda$, there is a surjective quandle homomorphism $p:\tilde{X}\to X$ associated with the extension.
Then we can verify that $p:\tilde{X}\to X$ is a covering. 

A \textit{pointed quandle} $(X,x)$ is a quandle $X$ with a chosen base element $x$ of $X$. A map $f:(X,x)\to (X^\prime,x^\prime)$ between pointed quandles is a {\it quandle homomorphism} (resp.~{\it covering}) if $f$ is a quandle homomorphism (resp.~covering) from $X$ to $X^\prime$ such that $f(x)=x^\prime$.
A quandle covering $f:(Q,q)\to (X,x)$ is said to be a \textit{universal covering} of $(X,x)$ if for any quandle covering $p:(\tilde{X},\tilde{x})\to (X,x)$, there exists a unique quandle homomorphism $\tilde{f}: (Q,q) \to (\tilde{X},\tilde{x})$ such that $p\circ\tilde{f}=f$. 
 
The \textit{fundamental group} of a quandle $X$ based at an element $x$ of $X$, 
which was introduced in \cite{Eisermann2014quandle}, 
is defined by $\{g\in \textrm{As}(X)^{\prime}\mid x\cdot g=x\}$, where $\textrm{As}(X)^{\prime}$ is the commutator subgroup of $\textrm{As}(X)$. 
The fundamental group of $X$ based at $x$ is denoted by $\pi_1(X,x)$. 
We note that $\pi_1(X,x)$ does not depend on the choice of the base point $x$ whenever $X$ is connected, and that $\pi_1(X,x)$ is sometimes denoted by $\pi_1(X)$ for a connected quandle $X$. 
A quandle $X$ is said to be \textit{simply connected} if $X$ is connected and $\pi_1(X,x)$ is trivial for some $x$ in $X$. 
%
We recall here 
\cite[Proposition~5.15]{Eisermann2014quandle}, which characterizes the simply connectedness of quandles in terms of universal coverings. 
Since the treatment of base points of pointed quandles is somewhat ambiguous in \cite[Proposition~5.15]{Eisermann2014quandle}, we modify his condition appropriately as follows. 

\begin{prop}\label{prop:simply_connect_equivalent}
The following are equivalent for a quandle $Q$. 
\begin{itemize}
\item[$(1)$] The quandle $Q$ is simply connected.
\item[$(2)$] 
There exists an element $q$ in $Q$ such that
any covering from $(Q,q)$ to any pointed quandle $(X,x)$ 
is a universal covering of $(X,x)$. 
\end{itemize}
\end{prop}

\subsection{Coverings and the types of quandles}\label{subsect:type}
We investigate a relation between the type of a quandle and that of another quandle, when there is a quandle covering from one to the other. 

\begin{lemm}
\label{lemm:type_covering}
For quandles $X$ and $\tilde{X}$, 
suppose that there is a surjective quandle homomorphism $p:\tilde{X}\to X$. 
If the type of $\tilde{X}$ is finite, 
then  the type of $X$ is a divisor of that of $\tilde{X}$. 
\end{lemm}
\begin{proof}
For each two elements $a,b$ in $X$, we have $p^{-1}(\tilde{a}) \neq \emptyset$ and $p^{-1}(\tilde{b}) \neq \emptyset$ by the assumption, and choose two elements $\tilde{a},\tilde{b}$ in $\tilde{X}$ such that 
$\tilde{a}\in p^{-1}(a)$ and $\tilde{b}\in p^{-1}(b)$. Then we have 
\[
a\ast^{\tilde{m}} b=p(\tilde{a})\ast^{\tilde{m}}p(\tilde{b})=p(\tilde{a}\ast^{\tilde{m}}\tilde{b})=p(\tilde{a})=a,
\]
where $\tilde{m}$ denotes the type of $\tilde{X}$. The equation $a\ast^{\tilde{m}}b=a$ implies the assertion.
\end{proof}

\begin{prop}
\label{prop:type_covering}
For quandles $X$ and $\tilde{X}$, 
suppose that there is a quandle covering $p:\tilde{X}\to X$. 
If the type of $X$ is finite and $X$ is connected, 
then the type of $\tilde{X}$ is equal to that of $X$. 
\end{prop}
\begin{proof}
For an element $x \in X$, 
let 
$\bar{p}:(\bar{X},\bar{x})\to (X,x)$ 
be the universal covering of the connected pointed quandle $(X,x)$, where $\bar{X}$ is also connected by \cite[Lemma~5.2 and Theorem~5.3]{Eisermann2014quandle}. 
For notational simplicity, we denote $\mathrm{type}(X)$, $\mathrm{type}(\tilde{X})$ and $\mathrm{type}(\bar{X})$ by $m$, $\tilde{m}$ and $\bar{m}$, respectively. 

We show that $\bar{m} \mid m$. 
For any elements $\bar{a}, \bar{b} \in \bar{X}$, we have $\bar{p}(\bar{a}\ast^m \bar{b})=\bar{p}(\bar{a})\ast^m \bar{p}(\bar{b})=\bar{p}(\bar{a})$. 
This implies that the map $S^m_{\bar{b}}:\bar{X}\to \bar{X}$ is a quandle homomorphism satisfying $\bar{p}=\bar{p} \circ S^m_{\bar{b}}$. 
By \cite[Corollary~4.11]{Eisermann2014quandle}, we see that $S^m_{\bar{b}}={\rm Id}_{\bar{X}}$, 
since $S^m_{\bar{b}}$ stabilizes the element $\bar{b} \in \bar{X}$ and $\bar{X}$ is connected. 
Thus we have $\bar{m} \mid m$ and in particular $\bar{m}<\infty$.

We set $I:=p^{-1}(x)$. We define a quandle operation $I\times\bar{X}$ by $(i,\bar{a})\ast(j,\bar{b})=(i,\bar{a}\ast\bar{b})$ for $i,j\in I$ and $\bar{a}, \bar{b} \in \bar{X}$. Then by \cite[Proposition 5.6 and Lemma~6.4]{Eisermann2014quandle}, the fundamental group $\pi_1(X,x)$ acts on $\bar{X}$ from the left, and on $I$ from the right. We consider an equivalence relation, denoted by $\sim$, on $I\times\bar{X}$ defined by $(i\cdot g,\bar{a})\sim(i,g\cdot\bar{a})$ for $i\in I,\bar{a}\in\bar{X}$ and $g\in\pi_1(X,x)$. Then the quotient set $I\times\bar{X}/\sim$ has a quandle operation inherited from $I\times\bar{X}$. It follows from \cite[Theorem 6.6]{Eisermann2014quandle} that the quandle $\tilde{X}$ is quandle isomorphic to $I\times\bar{X}/\sim$. This implies that there is a surjective quandle homomorphism $f:I\times\bar{X}\to\tilde{X}$.

We can easily see that ${\rm type}(I\times \bar{X})={\rm type}(\bar{X}) = \bar{m}$ by the definition of the quandle $I\times \bar{X}$. 
It follows from $\bar{m} < \infty$ that Lemma~\ref{lemm:type_covering} can be applied for the surjective quandle homomorphism $f \colon I\times \bar{X} \to \tilde{X}$. 
Then it holds that $\tilde{m} \mid \bar{m}$, and further by $\bar{m} \mid m$ we have $\tilde{m} \mid m$ and in particular $\tilde{m}<\infty$.
It also follows from $\tilde{m}<\infty$ that Lemma~\ref{lemm:type_covering} can be applied for the covering $p:\tilde{X}\to X$. 
Then it holds that $m \mid \tilde{m}$, and further by $\tilde{m} \mid m$ we have $\tilde{m} = m$. 
 \end{proof}

\begin{cor}\label{cor:type_covering}
For quandles $X$ and $\tilde{X}$, 
suppose that there is a quandle covering $p:\tilde{X}\to X$. 
If the type of $\tilde{X}$ is finite and $\tilde{X}$ is connected, 
then the type of $X$ is equal to that of $\tilde{X}$. 
\end{cor}

\begin{proof}
It follows from Lemma~\ref{lemm:type_covering} that the type of $X$ is finite. 
Since $p$ is surjective and the source $\tilde{X}$ of $p$ is connected, the target $X$ of $p$ is also connected. 
Then we can apply Proposition~\ref{prop:type_covering} for the covering $p$. 
\end{proof}

\subsection{Universal coverings of the knot $n$-quandles}
\label{subsect:universal_covering_and_simply_connect}
For an oriented $1$-knot $K$ and an integer $n>1$, 
we show the following proposition needed to prove that the quandle covering associated with the central extension $A^n_K \curvearrowright Q(\tau^n K) \to Q_n(K)$ is a universal covering.

\begin{theo}
\label{theo:simply_connect_twist_spun}
The knot quandle $Q(\tau^nK)$ of the $2$-knot $\tau^nK$ is simply connected for an oriented $1$-knot $K$ and an integer $n>1$.
\end{theo}
\begin{proof}
Let $D$ be a diagram of the long knot obtained by cutting $K$ open. 
Let us recall the presentation of $Q(\tau^n K)$ obtained from $D$; see \cite{Satoh2002surface}. 
Traveling along $D$ from one end point to the other end point according to the orientation of $D$, 
we label the arcs of $D$ consecutively from $a_0$ to $a_m$, where $m$ is the number of the crossings of $D$. 
For each $i$, let $\chi_i$ be the crossing of $D$ dividing $a_{i-1}$ and $a_{i}$, and $\varepsilon_i$ the sign of the crossing $\chi_i$. We denote by $\kappa_i$ the subscript of the label of the over arc at the crossing $\chi_i$, that is, the label of that over arc is $a_{\kappa_i}$ for each $i$. We define the relator $r_i$ by $a_{i-1}\ast^{\varepsilon_i}a_{\kappa_i}=a_{i}$. Then the knot quandle $Q(\tau^n K)$ has a following presentation:
\[
\langle a_0,a_1,\ldots,a_{m}\mid r_1,\ldots,r_m, a_i\ast^n a_0=a_i\ (i=1,\ldots,m)\rangle.
\]

By Proposition~\ref{prop:simply_connect_equivalent}, in order to prove that $Q(\tau^n K)$ is simply connected, it is sufficient to show that any covering from $(Q(\tau^nK),a_0)$ to any pointed quandle $(X,x)$ is a universal covering of $(X,x)$. 
Let $f:(Q(\tau^nK),a_0)\to (X,x)$ and $p:(\tilde{X},\tilde{x})\to (X,x)$ be two coverings. 
It follows from the relator $r_i$ of the presentation of $Q(\tau^nK)$ that we have $f(a_{i-1}) *^{\varepsilon_i} f(a_{\kappa_i}) = f(a_i)$ for all $i$. 
We define a map $\widetilde{f}:\{ a_0,a_1,\ldots,a_{m}\}\to \tilde{X}$ as follows.  First, we set $\widetilde{f}(a_0):=\tilde{x}$. Since $p$ is a covering, we can set $\widetilde{f}(a_{i}):=\widetilde{f}(a_{i-1})\ast^{\varepsilon_i} f(a_{\kappa_i})$ at each crossing $\chi_i$ inductively; recall the first paragraph of Subsection~\ref{subsect:def_covering_and_simply_connect}. Since we have 
$p ( \widetilde{f}(a_{i-1})\ast^{\varepsilon_i} f(a_{\kappa_i}) ) 
= p((\widetilde{f}(a_{i-1})) \ast^{\varepsilon_i} f(a_{\kappa_i}),$
it holds that $p(\widetilde{f}(a_i))=f(a_i)$ for all $i$ by induction. 
Moreover, it follows from $p(\widetilde{f}(a_i))=f(a_i)$ 
that we have $\widetilde{f}(a_{i})=\widetilde{f}(a_{i-1})\ast^{\varepsilon_i}\widetilde{f}(a_{\kappa_i})$ for all $i$.

Since the type of $Q(\tau^n K)$ is equal to $n$ by \cite[Theorem~3.3]{TanakaTaniguchi}, we see that ${\rm type}(X) \mid n$ by Lemma~\ref{lemm:type_covering} for the covering $f$. Since $f$ is surjective and the source $Q(\tau^n K)$ of $f$ is connected, the target $X$ of $f$ is also connected. Then, by Proposition~\ref{prop:type_covering} for the covering $p$, it holds that ${\rm type}(\tilde{X})={\rm type}(X)$ and hence we have ${\rm type}(\tilde{X}) \mid n$. 
Thus, it holds that $\widetilde{f}(a_i)\ast^n\widetilde{f}(a_0)=\widetilde{f}(a_i)$ for all $i$. 
This implies that the map $\widetilde{f}:\{ a_0,a_1,\ldots,a_{m}\}\to \tilde{X}$ can be uniquely extended to the quandle homomorphism $\widetilde{f}:(Q(\tau^n K),a_0)\to (\tilde{X},\tilde{x})$ such that $p\circ\widetilde{f}=f$.
\end{proof}

%
%
Using Proposition~\ref{prop:simply_connect_equivalent} and Theorem~\ref{theo:simply_connect_twist_spun} for the covering associated with the central extention in Theorem~\ref{theo:central_extension}, we have the following: 

\begin{theo}
\label{theo:universal_covering}
The covering $p \colon Q(\tau^nK)\to Q_n(K)$ associated with the central extension 
$A^n_K \curvearrowright Q(\tau^n K) \to Q_n(K)$ is  a universal covering of $(Q_n(K),x)$ for any $x$ in $Q_n(K)$ for an oriented $1$-knot $K$ and an integer $n>1$. 
\end{theo}


\subsection{Induced group homomorphisms}
We observe the induced group homomorphism $p_{\#}:{\rm As}(Q(\tau^n K))\to {\rm As}(Q_n(K))$ in 
Proposition~\ref{prop:induce}, which will not be used in this paper but has its own interest. 
Although it follows from \cite[Example~1.25]{Eisermann2014quandle} that there exists a universal covering $f:(\tilde{X},\tilde{x})\to (X,x)$ 
such that 
$\mathrm{As}(\tilde{X})$ is group isomorphic to $\mathrm{As}(X)$
but 
the induced group homomorphism $f_{\#}:{\rm As}(\tilde{X})\to {\rm As}(X)$ is not a group isomorphism, 
we have the following for the universal covering $p \colon Q(\tau^nK)\to Q_n(K)$. 

\begin{prop}\label{prop:induce}
The map $p_{\#}:{\rm As}(Q(\tau^n K))\to {\rm As}(Q_n(K))$ induced from the universal covering $p \colon Q(\tau^nK)\to Q_n(K)$ is a group isomorphism.
\end{prop}

In order to prove Proposition~\ref{prop:induce}, we need the following:  

\begin{prop}
\label{prop:simply_connect_associated_group}
Let $p:(\tilde{X},\tilde{x})\to (X,x)$ be a covering between pointed quandles. 
If $X$ is connected and $p$ is a universal covering, 
then the following are equivalent$:$ 
\begin{itemize}
\item The quandle $\tilde{X}$ is simply connected.
\item The induced map $p_{\#}:{\rm As}(\tilde{X})\to {\rm As}(X)$ is a group isomorphism.
\end{itemize}  
\end{prop}
\begin{proof}
We remark that if a quandle homomorphism $f:X\to Y$ is surjective, the induced group homomorphism $f_{\#}:{\rm As}(X)\to{\rm As}(Y)$ is also surjective. It is shown in  \cite[Proposition~5.10]{Eisermann2014quandle} that each covering $p:(\tilde{X},\tilde{x})\to (X,x)$ induces the group homomorphism $p_{\ast}:\pi_1(\tilde{X},\tilde{x})\to \pi_1(X,x)$ and that ${\rm Ker}(p_{\#}:{\rm As}(\tilde{X})\to{\rm As}(X))$ is equal to ${\rm Ker}(p_{\ast})$.

By \cite[Theorem~5.22 and Proposition~5.23]{Eisermann2014quandle}, 
we see that the induced group homomorphism $p_{\ast}:\pi_1(\tilde{X},\tilde{x})\to \pi_1(X,x)$ is the $0$-map, since $X$ is connected and $p$ is universal. 
This implies that ${\rm Ker}(p_{\ast})$ is equal to $\pi_1(\tilde{X},\tilde{x})$. 
Hence, 
the quandle $\tilde{X}$ is simply connected if and only if $p_{\#}$ is a group isomorphism.
\end{proof}

\begin{proof}[Proof of Proposition~\ref{prop:induce}]
By Theorem~\ref{theo:universal_covering}, the covering $p:Q(\tau^n K) \to Q_n(K)$ is a universal covering. Since the knot quandle $Q(\tau^n K)$ is simply connected, it holds that the induced group homomorphism $p_{\#}:{\rm As}(Q(\tau^n K))\to {\rm As}(Q_n(K))$ is a group isomorphism by Proposition~\ref{prop:simply_connect_associated_group}. 
\end{proof}

\section{The second quandle homology group of an $n$-quandle}
\label{sect:quandle_homology}

We compute the second quandle homology group $H^Q_2(Q_n(K))$ of the knot $n$-quandle $Q_n(K)$ for each oriented $1$-knot $K$ and each integer $n>1$. 
Instead of computing according to the original definition in \cite{Carter2003quandle}, we compute it by using 
the central extension $A^n_K \curvearrowright Q(\tau^n K) \to Q_n(K)$,  
%
where $A^n_K$ is the cyclic subgroup of 
$\pi_1(M^n_K)$ 
generated by the preferred longitude $l_K$ of $K$ as in Section~\ref{sect:central_extension}. 
%
We recall here 
\cite[Proposition~5.9]{Eisermann2014quandle}, which connects an extension of $X$ with the fundamental group of $X$ for a connected quandle $X$: 
\begin{prop}\label{prop:Eisermann_ext-and-pi1}
For a connected quandle $X$, 
if a quandle $\tilde{X}$ is an extension of $X$ by a group $\Lambda$, then the associated covering $p:\tilde{X}\to X$ induces a natural group homomorphism $h:\pi_1(X,x)\to \Lambda$ for any $x\in X$. In particular, $h$ is a group isomorphism if and only if $p$ is a universal covering. 
\end{prop}
By this proposition, we have the following:
\begin{prop}
\label{prop:second_quandle_homology}
The fundamental group $\pi_1(Q_n(K),x)$ is group isomorphic to $A^n_K$   
for an oriented $1$-knot $K$ and an integer $n>1$.
\end{prop}
\begin{proof}
Theorem~\ref{theo:central_extension} states that $Q(\tau^nK)$ is an extension of the connected quandle $Q_n(K)$ by $A^n_K$. Since the associated covering $p:Q(\tau^nK)\to Q_n(K)$ is universal by Theorem~\ref{theo:universal_covering}, 
it follows from Proposition~\ref{prop:Eisermann_ext-and-pi1} that 
$\pi_1(Q_n(K),x)$ is group isomorphic to $A^n_K$ for any $x\in Q_n(K)$.
\end{proof}
It is known in \cite[Theorem~1.15]{Eisermann2014quandle} that the second quandle homology group $H^Q_2(X)$ of $X$ is group isomorphic to the abelianization of the fundamental group $\pi_1(X,x)$ for a connected pointed quandle $(X,x)$. 
Hence we have: 
%
\begin{cor}
\label{cor:homology_n-quandle}
The second quandle homology group $H^Q_2(Q_n(K))$ is group isomorphic to $A^n_K$ 
for an oriented $1$-knot $K$ and an integer $n>1$.
\end{cor}


By Corollary~\ref{cor:homology_n-quandle}, in order to compute $H^Q_2(Q_n(K))$, it is sufficient to determine the order of $l_K$, which is the generator of the cyclic subgroup $A^n_K$ of $\pi_1(M^n_K)$. 
We discuss the order of $l_K$ when $K$ is prime in Subsection~\ref{subsec:case_prime_knot} and 
when $K$ is composite in Subsection~\ref{subsec:case_composite_knot}, respectively.

\subsection{The case of prime knots}
\label{subsec:case_prime_knot}
We compute the order of $l_K$ in $\pi_1(M^n_K)$ when $K$ is prime. 
First, we discuss the case where $\pi_1(M^n_K)$ is finite. 
Next, we discuss the case where $\pi_1(M^n_K)$ is infinite.

\subsubsection{The case where $\pi_1(M^n_K)$ is finite}
We compute the order of $l_K$ in $\pi_1(M^n_K)$ when $\pi_1(M^n_K)$ is finite. 
By the elliptization theorem, the group $\pi_1(M^n_K)$ is finite if and only if the universal covering space of $M^n_K$ is $S^3$; see \cite[Subsection~1.7]{Aschenbrenner2015}. 
According to the Dunbar's classification \cite{Dunbar1988geometric} of spherical orbifolds, we have the following: 

\begin{lemm}\label{lemm:finite}
For an oriented prime $1$-knot $K$ and an integer $n>1$, 
if the universal covering space of $M^n_K$ is $S^3$, then one of the following holds:
 \begin{itemize}
 \item[$(1)$] A knot $K$ is a $2$-bridge knot $S(\alpha,\beta)$ and $n=2$.
 \item[$(2)$] A knot $K$ is a Montesinos knot $M(b;\frac{1}{2},\frac{\beta_2}{3},\frac{\beta_3}{3})$ and $n=2$.
 \item[$(3)$] A knot $K$ is a Montesinos knot $M(b;\frac{1}{2},\frac{\beta_2}{3},\frac{\beta_3}{5})$ and $n=2$.
 \item[$(4)$] A knot $K$ is the trefoil $3_1$ and $n=3,4,5$.
 \item[$(5)$] A knot $K$ is the cinquefoil $5_1$ and $n=3$.
 \end{itemize}  
\end{lemm}
 
%

Here we follow convention of the textbook \cite{Burde2014knots} on knot theory for notations of a $2$-bridge knot $S(\alpha,\beta)$ and a Montesinos knot $M(b;\frac{\beta_1}{\alpha_1},\frac{\beta_2}{\alpha_2},\frac{\beta_3}{\alpha_3})$,  where the integers $\alpha, \alpha_i, \beta$ and $\beta_i$ satisfy $\alpha \geq 2, \alpha_i \geq 2$ and ${\rm gcd}(\alpha,\beta)={\rm gcd}(\alpha_i,\beta_i)=1$ for $i=1,2,3$. 
%


\begin{prop}\label{prop:finite}
For an oriented prime $1$-knot $K$ and an integer $n>1$, 
if $\pi_1(M^n_K)$ is finite, 
then we have the following table for the order of $l_K$. 
%
\[\begin{array}{c || c  c c | c c c | c}
\toprule 
\text{\rm Lemma~\ref{lemm:finite}} & (1) & (2) & (3) & & (4) & & (5) \\ 
\midrule   
K & S(\alpha,\beta) & M(b;\frac{1}{2},\frac{\beta_2}{3},\frac{\beta_3}{3}) &
M(b;\frac{1}{2},\frac{\beta_2}{3},\frac{\beta_3}{5}) &
3_1 & 3_1 & 3_1 & 5_1 \\
n & 2 & 2 & 2 & 3 & 4 & 5 & 3 \\
\midrule 
\text{\rm order of} \ l_K & 1 & 2 & 4 & 2 & 4 & 10 & 6 \\ 
\bottomrule 
\end{array}\]
%
\end{prop}

\begin{proof}
First, we discuss the cases (4) and (5) in Lemma~\ref{lemm:finite}. 

\medskip
\underline{Case (4).} According to \cite{Inoue2021knot}, if $n$ is 3, 4 or 5, then it holds that the order of $l_{3_1}$ is equal to 2, 4 or 10, respectively in $\pi_1(M^n_{3_1})$. 

\medskip
\underline{Case (5).} It is shown in \cite{Crans2019finite} that the order of $Q_3(5_1)$ is equal to $20$. 
By Proposition~\ref{prop:knot_n_qdle} and Corollary~\ref{cor:homology_n-quandle}, the group $H^Q_2(Q_3(5_1))$ is the cyclic group whose order is equal to $|Q(\tau^3(5_1))|/|Q_3(5_1)|$.  Since $|Q(\tau^3(5_1))|=|\pi_1(M^3_{5_1})|=120$, it holds that $H^Q_2(Q_3(5_1))$ is group isomorphic to $\Z/6\Z$, which implies that the order of $l_{5_1}$ is equal to $6$ in $\pi_1(M^3_{5_1})$. 

\medskip
Next we discuss the cases (1), (2) and (3) in Lemma~\ref{lemm:finite}. 
In these cases, we have $n=2$ and the universal covering space of $M^2_K$ is $S^3$. 
%
Let $p:S^3\to M^2_K$ be the universal covering of $M^2_K$. Then, the inverse image $L:=p^{-1}(\tilde{K})$ of the branching set $\tilde{K}$ of $M^n_K$ is an oriented $1$-link in $S^3$, where the orientation of $L$ is induced by the orientation of $\tilde{K}$. 
For a component $L'$ of $L$, 
we see that the restriction map $p|_{L'}:L' \to \tilde{K}$ is a covering and that the degree of $p|_{L'}$ coincides with the order of $l_K = [\tilde{K}] \in\pi_1(M^2_K)$. Since $p$ is a covering, it holds that the degree of $p|_{L'}$ is equal to $|\pi_1(M^2_K)|/c_L$, where $c_L$ denote the number of the components of $L$.  
Thus, it is sufficient to compute the number $c_L$ for the oriented $1$-link $L=p^{-1}(\tilde{K})$ in $S^3$, where $L$ has been studied in detail by Sakuma in~\cite[Section~5]{Sakuma1990geometries}. 
 
\medskip
\underline{Case (1).} Suppose that $K$ is a 2-bridge knot $S(\alpha.\beta)$. 
Then we have 
$c_L = \alpha$ by \cite[Theorem~5.3]{Sakuma1990geometries}. 
Since $|\pi_1(M^2_K)| = \alpha$, we see that the order of $l_K$ is equal to $1$ in $\pi_1(M^2_K)$. 

\medskip
\underline{Case (2).} Suppose that $K$ is a Montesinos knot $M(b;\frac{1}{2},\frac{\beta_2}{3},\frac{\beta_3}{3})$. We set $\mu=|6(-b+\frac{1}{2}+\frac{\beta_2}{3}+\frac{\beta_3}{3})|$. 
Then we have 
$c_L = 12 \mu$ and $|\pi_1(M^2_K)| = 24\mu$ by \cite[Theorem~5.3 and Theorem~2.10]{Sakuma1990geometries}, respectively. 
Hence we see that the order of $l_K$ is equal to $2$ in $\pi_1(M^2_K)$. 

\medskip
\underline{Case (3).} Suppose that $K$ is a Montesinos knot $M(b;\frac{1}{2},\frac{\beta_2}{3},\frac{\beta_3}{5})$. We set $\mu=|30(-b+\frac{1}{2}+\frac{\beta_2}{3}+\frac{\beta_3}{5})|$. 
Then we have 
$c_L = 30 \mu$ and $|\pi_1(M^2_K)| = 120\mu$ by \cite[Theorem~5.3 and Theorem~2.10]{Sakuma1990geometries}, respectively. 
Hence we see that the order of $l_K$ is equal to $4$ in $\pi_1(M^2_K)$. 
\end{proof}

\subsubsection{The case where $\pi_1(M^n_K)$ is infinite}
We compute the order of $l_K$ in $\pi_1(M^n_K)$ when $\pi_1(M^n_K)$ is infinite.

\begin{lemm}
\label{lemm:universal_cover_of_branched_cover}
For an oriented prime $1$-knot $K$ and an integer $n$ greater than $1$,  
if the fundamental group $\pi_1(M^n_K)$ is infinite, 
then the universal covering space of $M^n_K$ is homeomorphic to $\R^3$. 
\end{lemm}
\begin{proof}
We note that by the equivalent sphere theorem \cite{Meeks1980topology}, the branched covering space $M^n_K$ is irreducible. See also \cite[Theorem~2]{Sakuma1982regular}. 

Suppose that $M^n_K$ contains an essential torus. Then we see that $M^n_K$ is a Haken manifold. 
Waldhausen \cite{Waldhausen1968irreducible} showed that the universal covering space of a Haken manifold is $\R^3$. Thus, the universal covering space of $M^n_K$ is $\R^3$. 

Suppose that $M^n_K$ does not contain an essential torus. According to the geometrization theorem, we see that $M^n_K$ admits one of the following eight geometries: $S^3,S^2\times\mathbb{E},\mathbb{E}^3,\mathbb{H}^3,\mathbb{H}^2\times\mathbb{E},{\rm Nil},{\rm Sol}$ and $\widetilde{{\rm SL}_2(\R)}$. 
%
%
We note that each underlying space of six geometries except the first two 
of the eight is $\R^3$. 
Thus, the universal covering space of $M^n_K$ is either $S^3,S^2\times\R$ or $\R^3$. Since $M^n_K$ is irreducible, we see that $\pi_2(M^n_K)=0$ by \cite{Papakyriakopoulos1957Dehn}. This implies that the universal covering space of $M^n_K$ is not homeomorphic to $S^2\times \R$. As we mentioned just before Lemma~\ref{lemm:finite}, if the universal covering space of $M^n_K$ is homeomorphic to $S^3$, the fundamental group $\pi_1(M^n_K)$ is finite. Hence, we see that the universal covering space of $M^n_K$ is $\R^3$.
\end{proof}
\begin{prop}
\label{prop:quandle_homology_case_infinite}
For an oriented prime $1$-knot $K$ and an integer $n$ greater than $1$,  
if the group $\pi_1(M^n_K)$  is infinite, 
then $l_K$ is not a torsion element in $\pi_1(M^n_K)$, that is, the order of $l_K$ is infinite in $\pi_1(M^n_K)$. 
\end{prop}

\begin{proof}
Since $M^n_K$ is irreducible and  $\pi_1(M^n_K)$ is infinite, we have that 
$\pi_1(M^n_K)$ is torsion free; see \cite[Subsection~3.2~(C.3)]{Aschenbrenner2015}. Thus, it is sufficient to show that $l_K$ is nontrivial in $\pi_1(M^n_K)$. Assume that $l_K$ is the identity element in $\pi_1(M^n_K)$. By Lemma~\ref{lemm:universal_cover_of_branched_cover}, let $p:\R^3\to M^n_K$ be the universal covering of $M^n_K$. By the assumption, the inverse image $p^{-1}(\tilde{K})$ of the branching set $\tilde{K}$ of $M^n_K$ is a $1$-link in $\R^3$, that is, each connected component of $p^{-1}(\tilde{K})$ is homeomorphic to a circle, not a line. 
For a fixed element $x$ in $p^{-1}(\tilde{K})$, let $\widetilde{\varphi}:\R^3\to \R^3$ be a lift of the canonical covering transformation $\varphi$ of $M^n_K$ 
such that $\widetilde{\varphi}(x)=x$, and $G$ a subgroup of $\langle \widetilde{\varphi}\rangle$ such that the order $|G|$ of $G$ is prime. 
We denote $\textrm{Fix}(G)$ the fixed point set of $\R^3$ by $G$. 
Since the connected component of $p^{-1}(\tilde{K})$ containing $x$ is fixed by $G$, 
the set $\textrm{Fix}(G)$ must contain a circle. 
It follows from Smith theory of group actions that $H_1(\textrm{Fix}(G);\F_{|G|})$ is trivial, where $\F_{|G|}$ is a finite field of prime order $|G|$; see \cite[Theorem~5.2 in Chapter III]{Bredon1972introduction} for example. 
This contradicts the fact that $\textrm{Fix}(G)$ contains a circle. 
Thus, we see that $l_K$ is nontrivial in $\pi_1(M^n_K)$.
\end{proof}

\subsection{The case of composite knots}
\label{subsec:case_composite_knot}
We discuss the order of $l_K$ in $\pi_1(M^n_K)$ when $K$ is composite. 
Before that, we summarize when $l_K$ is trivial in $\pi_1(M^n_K)$ for an oriented prime $1$-knot $K$. 
It follows from Proposition~\ref{prop:finite}~and \ref{prop:quandle_homology_case_infinite} that we have the following: 

\begin{cor}\label{cor:trivial}
The element $l_K$ for an oriented prime $1$-knot $K$ is trivial in $\pi_1(M^2_K)$ if and only if $K$ is $2$-bridge knot. 
The element $l_K$ for $K$ is always nontrivial in $\pi_1(M^n_K)$ for each interger $n>2$. 
\end{cor}

For oriented $1$-knots $K_1$ and $K_2$, let $K$ be the connected sum 
of $K_1$ and $K_2$. 
Then we see that the fundamental group $\pi_1(M^n_K)$ is group isomorphic to the free product $\pi_1(M^n_{K_1})\ast\pi_1(M^n_{K_2})$ for each integer $n>1$. 
Moreover, by the definition of $l_K$, it holds that $l_K=l_{K_1}\cdot l_{K_2}$ in $\pi_1(M^n_K)$. Hence, we have the following: 

\begin{lemm}\label{lemm:composite}
Let $K$ be the connected sum of oriented $1$-knots $K_1$ and $K_2$. 
%
If $l_{K_2}$ is trivial in $\pi_1(M^n_{K_2})$, then the order of $l_K$ in $\pi_1(M^2_K)$ 
is equal to that of $l_{K_1}$ in $\pi_1(M^2_{K_1})$. 
If $l_{K_1}$ and $l_{K_2}$ are nontrivial in $\pi_1(M^n_{K_1})$ and $\pi_1(M^n_{K_2})$ respectively, then the order of $l_K$ is infinite in $\pi_1(M^n_K)$ for each integer $n>2$.
\end{lemm}

Combining Lemma~\ref{lemm:composite} with Corollary~\ref{cor:trivial}, we have the following: 

\begin{prop}\label{prop:composite}
Let $K$ be an oriented composite $1$-knot. 
If $K$ is equivalent to an oriented $1$-knot $K'$ up to $2$-bridge knot summands, then the order of $l_K$ in $\pi_1(M^2_K)$ is equal to that of $l_{K'}$ in $\pi_1(M^2_{K'})$. 
If an integer $n$ is greater than $2$, then the order of $l_K$ is always infinite in $\pi_1(M^n_K)$. 
\end{prop}

\subsection{Second quandle homology groups}\label{subsect:compute}


Summarizing the results (Proposition~\ref{prop:finite}, ~\ref{prop:quandle_homology_case_infinite} and ~\ref{prop:composite}), 
we can compute $H^Q_2(Q_n(K))$ completely as follows. 

\begin{theo}\label{theo:2-quandle} 
The following hold for the knot $2$-quandles. 
\begin{itemize}
\item If an oriented $1$-knot $K$ is equivalent to 
the unknot up to $2$-bridge knot summands, then $H^Q_2(Q_2(K))$ is trivial.
\item If $K$ is equivalent to 
a Montesinos knot $M(b;\frac{1}{2},\frac{\beta_2}{3},\frac{\beta_3}{3})$ up to $2$-bridge knot summands, then $H^Q_2(Q_2(K))$ is group isomorphic to $\Z/2\Z$.
\item If $K$ is equivalent to 
a Montesinos knot $M(b;\frac{1}{2},\frac{\beta_2}{3},\frac{\beta_3}{5})$ up to $2$-bridge knot summands, then $H^Q_2(Q_2(K))$ is group isomorphic to $\Z/4\Z$.
\item Otherwise, it holds that $H^Q_2(Q_2(K))$ is group isomorphic to $\Z$. 
\end{itemize}
\end{theo}

\begin{theo}\label{theo:3-quandle}  
The following hold for the knot $3$-quandles.
\begin{itemize}
\item If an oriented $1$-knot $K$ is 
the unknot, then $H^Q_2(Q_3(K))$ is trivial.
\item If $K$ is 
the trefoil, then $H^Q_2(Q_3(K))$ is group isomorphic to $\Z/2\Z$.
\item If $K$ is 
the cinquefoil, then $H^Q_2(Q_3(K))$ is group isomorphic to $\Z/6\Z$.
\item Otherwise, it holds that $H^Q_2(Q_3(K))$ is group isomorphic to $\Z$. 
\end{itemize}
\end{theo}

\begin{theo}\label{theo:4-quandle}  
The following hold for the knot $4$-quandles. 
\begin{itemize}
\item If an oriented $1$-knot $K$ is 
the unknot, then $H^Q_2(Q_4(K))$ is trivial.
\item If $K$ is 
the trefoil, then $H^Q_2(Q_4(K))$ is group isomorphic to $\Z/4\Z$.
\item Otherwise, it holds that $H^Q_2(Q_4(K))$ is group isomorphic to $\Z$. 
\end{itemize}
\end{theo}

\begin{theo}\label{theo:5-quandle}  
The following hold for the knot $5$-quandles. 
\begin{itemize}
\item If an oriented $1$-knot $K$ is 
the unknot, then $H^Q_2(Q_5(K))$ is trivial.
\item If $K$ is 
the trefoil, then $H^Q_2(Q_5(K))$ is group isomorphic to $\Z/10\Z$.
\item Otherwise, it holds that $H^Q_2(Q_5(K))$ is group isomorphic to $\Z$. 
\end{itemize}
\end{theo}

\begin{theo}\label{theo:6-quandle} 
The following hold for the knot $n$-quandles for each $n>5$. 
\begin{itemize}
\item If an oriented $1$-knot $K$ is 
the unknot, then $H^Q_2(Q_n(K))$ is trivial.
\item Otherwise, it holds that $H^Q_2(Q_n(K))$ is group isomorphic to $\Z$. 
\end{itemize}
\end{theo}

\subsection{Characterizations of some knots}\label{subsect:detect}
Eisermann \cite{Eisermann2003unknot} showed that the second quandle homology group of the knot quandle characterizes the unknot. As a consequence of our results (Theorem~\ref{theo:2-quandle}, \ref{theo:3-quandle}, \ref{theo:4-quandle}, \ref{theo:5-quandle} and \ref{theo:6-quandle}), an analogous result 
holds for the knot $n$-quandle when $n>2$. 
%
Moreover, we can characterize some nontrivial $1$-knots for $n=3,4$ and $5$.

\begin{theo}\label{theo:detect}
The second quandle homology group of the knot $n$-quandle characterizes 
$(1)$ the unknot, the trefoil and the cinquefoil for $n=3$, 
$(2)$ the unknot and the trefoil for $n=4$ or $5$,  
$(3)$ the unknot for $n>5$, and  
$(4)$ the unknot up to $2$-bridge knot summands for $n=2$. 
\end{theo}

\section{Related results}\label{sect:topics}

In the final section, we collect the three related results: a relation of our extension in Theorem~\ref{theo:central_extension} with the extension in \cite{Inoue2021knot}, the types of the knot $n$-quandles, and the cardinalities of the finite knot $n$-quandles.   

\subsection{Relation with Inoue's result}
\label{appendix:Inoue_result}
Inoue \cite{Inoue2021knot} introduced the notion of the Schl\"{a}fli quandles and showed that the knot quandle $Q(\tau^n 3_1)$ of the $n$-twist spun trefoil is a central extension of the Schl\"{a}fli quandle $\{3, n\}$ by a cyclic group $A_n$ for each $n>1$. 
He also computed the cyclic group $A_n$ for $n \leq 5$ and asked what $A_n$ is for $n>5$. 
We confirm that the central extension in Theorem~\ref{theo:central_extension} can be considered as a generalization of his extension, and also give an answer to his question. 

We review the definition of the Schl\"{a}fli quandle. 
Let us consider the regular tessellation $\{ 3,n\}$ (of the spherical $2$-space for $2 \leq n \leq 5$, the Euclidean $2$-space for $n=6$ and the hyperbolic $2$-space for $n \geq 7$) in the sence of Schl\"{a}fli symbols. 
We denote the set of all vertices of $\{ 3,n\}$ by $X_n$. The set $X_n$ is a quandle with an operation by $v\ast w:=r_w(v)$, where $r_w$ is the rotation about $w$ by the angle $2\pi /n$. We call the quandle $X_n$ the \textit{Schl\"{a}fli quandle} related to $\{ 3,n\}$. 
Here we find a relation of $X_n$ with $Q_n(3_1)$. 

\begin{lemm}\label{lemm:Schlafli}
The Schl\"{a}fli quandle $X_n$ is quandle isomorphic to the knot $n$-quandle $Q_n(3_1)$ of the trefoil $3_1$.
\end{lemm}

\begin{proof}
By \cite[Lemma~6.4]{Inoue2021knot}, the Schl\"{a}fli quandle $X_n$ has a presentation 
\[
\langle v,w\mid  (v\ast w)\ast v=w,  (w\ast v)\ast w=v,  w\ast^nv=w\rangle.
\]
Since $r_w$ is the rotation by $2\pi /n$, it holds that $v\ast^nw=v$; 
see also Remark~\ref{remark:relator} below.  
Thus, the Schl\"{a}fli quandle $X_n$ has a presentation 
\[
\langle v,w\mid (v\ast w)\ast v=w, (w\ast v)\ast w=v, w\ast^nv=w, v\ast^n w=v\rangle.
\] 
Since $\langle v,w\mid (v\ast w)\ast v=w, (w\ast v)\ast w=v\rangle$ is a presentation of the knot quandle $Q(3_1)$, the second presentation of $X_n$ is equal to a presentation of the knot $n$-quandle $Q_n(3_1)$. 
\end{proof}

\begin{remark}\label{remark:relator}
Although we show $v *^n w = v$ in the proof of Lemma~\ref{lemm:Schlafli} by a geometric property of the rotation $r_w$, we can also prove it as a consequence obtained from 
the three relators 
$(v\ast w)\ast v=w$,  $(w\ast v)\ast w=v$ and $w\ast^nv=w$ of the first presentation of $X_n$.  
In fact, we have 
\begin{equation*}
\begin{split}
v *^n w & = v *^n ( (v*w)*v ) = (v*v) *^n ( (v*w)*v ) = \Big( v *^n (v*w) \Big) * v \\
& = \Big( ( (v *^{-1} w) * w) *^n (v*w) \Big) * v = 
\Big( ( (v *^{-1} w) *^n v )*w \Big) *v \\
& = 
\Big( \big( (v *^n v) *^{-1} (w *^n v) \big)*w \Big)*v =
\Big( (v *^{-1} w)  *w \Big)*v  \\ 
& = v * v  = v, 
\end{split}
\end{equation*}
where the relator $(v\ast w)\ast v=w$ implies the first equality, and the relator $w\ast^nv=w$ implies the third equality from the end. 
We note that the relator $(w\ast v)\ast w=v$ is not used in the computation. 
\end{remark}

It follows from Lemma~\ref{lemm:Schlafli} that the extension $A^n_{3_1} \curvearrowright Q(\tau^n 3_1) \to Q_n(3_1)$ in Theorem~\ref{theo:central_extension} for the trefoil $3_1$ is essentially the same as his extension $A_n \curvearrowright Q(\tau^n \, 3_1) \to X_n$. 
Hence, by Proposition~\ref{prop:quandle_homology_case_infinite}, we have that the cyclic group $A_n$ is group isomorphic to $\Z$ for $n>5$, answering Inoue's question. Here we give an alternative direct proof.  

\begin{proof}[A direct proof of $A^n_{3_1} \cong \Z$ for $n>5$]
Since $M^n_{3_1}$ is an irreducible 3-manifold with infinite fundamental group for $n>5$, the group $\pi_1(M^n_{3_1})$ is torsion-free; see \cite[Subsection~3.2~(C.3)]{Aschenbrenner2015}. Thus, it is sufficient to show that $l_{3_1}$ is not equal to the identity element $e$ in $\pi_1(M^n_{3_1})$. 
It is known in \cite{rolfsen1976knots} that 
$\pi_1(M^n_{3_1})$ has the following presentation 
\[
\langle x_1,\ldots,x_n\mid x_{i-1}=x_ix_{i-2}\ (i=1,\ldots,n)\rangle, 
\]
where the indices of the relations are taken modulo $n$. 
It is also known in \cite{Inoue2021knot} that $l_{3_1}=x_{i} x_{i-1}^{-1} x_{i}^{-1} x_{i-1}$ in $\pi_1(M^n_{3_1})$ for each $i$. 

Assume that $l_{3_1}=e$ in $\pi_1(M^n_{3_1})$. For any $i$, we have
\[
x_ix_{i-1}^{-1}x_i^{-1}x_{i-1}=e \Leftrightarrow x_ix_{i-1}^{-1}=x_{i-1}^{-1}x_i \Leftrightarrow x_ix_{i-1}=x_{i-1}x_i
\]
and
\[
x_i=x_{i+1}x_{i-1}\Leftrightarrow x_i=x_{i+2}x_ix_{i-1}\Leftrightarrow x_i=x_{i+2}x_{i-1}x_i\Leftrightarrow e=x_{i+2}x_{i-1}.
\]
Furthermore, it holds that
\[
x_i=x_{i+1}x_{i-1}\Leftrightarrow x_{i-1}x_ix_{i-1}^{-1}=x_{i-1}x_{i+1}\Leftrightarrow x_i=x_{i-1}x_{i+1}.
\]
It follows that $x_i$ commutes with $x_{j}$ for any $i$ and $j$, and hence $\pi_1(M^n_{3_1})$ is abelian.
Then the fundamental group $\pi_1(M^n_{3_1})$ is group isomorphic to the first homology group $H_1(M^n_{3_1})$. 
If $n$ is equal to $6k,6k\pm1,6k\pm2$ or $6k-3$, then it is known in \cite[Section~10.D]{rolfsen1976knots} that $H_1(M^n_{3_1})$ is group isomorphic to $\Z^2,0,\Z/3\Z$ or $(\Z/2\Z)^2$, respectively. On the other hand, for $n > 5$, it is known in \cite{Aschenbrenner2015} that $\pi_1(M^n_{3_1})$ is infinite and is not group isomorphic to $\Z^2$. This is a contradiction. 
\end{proof}

\subsection{Types of the knot $n$-quandles}
\label{subsect:type_n-quandle}
We discuss the types of the knot $n$-quandles. 
By definition, 
we can see that the type of $Q_n(K)$ is a divisor of $n$ for an oriented nontrivial $1$-knot $K$ and an integer $n>1$. 
Although the type of $Q_n(K)$ has not been determined exactly until now, 
we can prove the following: 
%
\begin{theo}\label{theo:type}
The type of $Q_n(K)$ is equal to $n$ for any oriented nontrivial $1$-knot $K$ and any integer $n>1$.  
\end{theo}
\begin{proof}
The authors proved in \cite[Theorem~3.3]{TanakaTaniguchi} that the type of the knot quandle $Q(\tau^n K)$ of $n$-twist spin $\tau^nK$ is equal to $n$. We recall that $Q(\tau^n K)$ is connected and there is a covering from $Q(\tau^nK)$ to $Q_n(K)$. By Corollary~\ref{cor:type_covering}, it holds that ${\rm type}(Q_n(K))={\rm type}(Q(\tau^nK))=n$.
\end{proof}

We give an alternative proof for \cite[Theorem~5.2.5]{Winker1984quandles} by Theorem~\ref{theo:type}. 

\begin{theo}\label{theo:Winker}
If the knot $n$-quandle $Q_n(K)$ of an oriented $1$-knot $K$ is trivial for some integer $n>1$, then $K$ is trivial. 
\end{theo}

\begin{proof}
If $K$ is nontrivial, then by Theorem~\ref{theo:type} we have that the type of $Q_n(K)$ is equal to $n$ for any integer $n>1$, and hence that $Q_n(K)$ is nontrivial for any integer $n>1$, which implies the contraposition. 
\end{proof}

Using Theorem~\ref{theo:Winker}, 
we give a proof for the following fact,  
where it may be folklore and a reference that contains its proof may not be found. 

\begin{theo}
The knot quandle $Q(K)$ is infinite for any oriented nontrivial $1$-knot $K$. 
Moreover, the type of $Q(K)$ is infinite. 
\end{theo}

\begin{proof}
If a quandle is finite, then type of the quandle is also finite, 
since the order of a bijection from the finite set to itself is finite. 
Thus it is sufficient to show the latter statement. 
Suppose that the type of $Q(K)$ is finite. 
Then, for a positive integer $n$ relatively prime to the type of $Q(K)$, 
it follows from the Euclidean algorithm that the knot $n$-quandle $Q_n(K)$ is trivial, 
and hence we have that $K$ is trivial by Theorem~\ref{theo:Winker}. 
This is a contradiction. 
\end{proof}

\subsection{The finite knot $n$-quandles}
We compute the cardinalities of all finite knot $n$-quandles completely. 
Hoste and Shanahan \cite{Hoste2017links} showed that the knot $n$-quandle $Q_n(K)$ is finite if and only if $\pi_1(M^n_K)$ is finite for an oriented $1$-knot $K$ and an integer $n>1$. 
Since the finite knot $n$-quandles of torus knots and $2$-bridge knots have already been studied in \cite{Crans2019finite}, 
it is sufficient to compute the cardinalities of the finite knot $2$-quandles of 
Montesinos knots except for $2$-bridge knots, that is, for 
Montesinos knots $M(b;\frac{1}{2},\frac{\beta_2}{3},\frac{\beta_3}{3})$ and Montesinos knots $M(b;\frac{1}{2},\frac{\beta_2}{3},\frac{\beta_3}{5})$; 
see Lemma~\ref{lemm:finite}. 
\begin{theo}
\label{theo:order_Montesinos_n-quandle}
The following hold for the finite knot $2$-quandles of Montesinous knots except for $2$-bridge knots. 
\begin{itemize}
\item[$(1)$] 
If $K$ is a Montesinos knot $M(b;\frac{1}{2},\frac{\beta_2}{3},\frac{\beta_3}{3})$, then 
we have that $|Q_2(K)|$ is equal to $12\mu_1$, where we set $\mu_1 =| 6 (-b+\frac{1}{2}+\frac{\beta_2}{3}+\frac{\beta_3}{3}) |$.
\item[$(2)$] 
If $K$ is a Montesinos knot $M(b;\frac{1}{2},\frac{\beta_2}{3},\frac{\beta_3}{5})$, then 
we have that $|Q_2(K)|$ is equal to $30\mu_2$, where we set $\mu_2 =| 30 (-b+\frac{1}{2}+\frac{\beta_2}{3}+\frac{\beta_3}{5}) |$.
\end{itemize}
\end{theo}
\begin{proof}
Since $Q_n(K)$ is quandle isomorphic to the quandle $A^n_K\backslash \pi_1(M^n_K)$, the order of $Q_n(K)$ is equal to $|\pi_1(M^n_K)|/|A^n_K|$ if $Q_n(K)$ is finite. 

We prove the case (1). 
Let $K$ be a Montesinos knot $M(b;\frac{1}{2},\frac{\beta_2}{3},\frac{\beta_3}{3})$. Then it follows from \cite[Theorem~2.10]{Sakuma1990geometries} that we have $|\pi_1(M^2_K)|=24\mu_1$. 
Since $|A^2_K| = 2$ by Proposition~\ref{prop:finite}, we have $|Q_2(K)|=12\mu_1$.

We prove the case (2). 
Let $K$ be a Montesinos knot $M(b;\frac{1}{2},\frac{\beta_2}{3},\frac{\beta_3}{5})$. Then it follows from \cite[Theorem~2.10]{Sakuma1990geometries} that we have $|\pi_1(M^2_K)|=120\mu_2$. Since $|A^2_K|=4$ by Proposition~\ref{prop:finite}, we have $|Q_2(K)|=30\mu_2$.
\end{proof}

The cardinalities of the finite knot $n$-quandles for torus knots and 2-bridge knots have already been computed in \cite{Crans2019finite}; see also \cite{Hoste2017links} for the trefoil. 
Combining their works and our result, we obtain the complete list of the cardinalities of all finite knot $n$-quandles for oriented $1$-knots as follows.
\[
\begin{array}{c | c c c c c c c}
\toprule
K & S(\alpha,\beta) & M(b;\frac{1}{2},\frac{\beta_2}{3},\frac{\beta_3}{3}) & 
M(b;\frac{1}{2},\frac{\beta_2}{3},\frac{\beta_3}{5}) &
3_1 & 3_1 & 3_1 & 5_1 \\
n & 2 & 2 & 2 & 3 & 4 & 5 & 3 \\
%
\midrule
|Q_n(K)| & \alpha & 12\mu_1 & 30\mu_2 & 4 & 6 & 12 & 20 \\ 
\midrule
\textrm{source} & \textrm{\cite{Crans2019finite}} & \textrm{Theorem~\ref{theo:order_Montesinos_n-quandle}} & \textrm{Theorem~\ref{theo:order_Montesinos_n-quandle}} &\textrm{\cite{Hoste2017links}} & \textrm{\cite{Hoste2017links}} & \textrm{\cite{Hoste2017links}} & \textrm{\cite{Crans2019finite}}\\
\bottomrule
\end{array}
\]

\section*{Acknowledgement}
The authors would like to thank Seiichi Kamada and Makoto Sakuma for their helpful comments. 
The first author was supported by JSPS KAKENHI Grant Number 17K05242 and 21K03220. 
The second author was supported by JSPS KAKENHI Grant Number 21J21482.

\bibliographystyle{plain}
\bibliography{reference}

\end{document}